\documentclass[11pt, a4paper]{article}
\usepackage{amsmath, amsthm, amssymb, url}
\usepackage[margin=1in]{geometry}
\usepackage[english]{babel}
\usepackage{multirow}
\usepackage{authblk}
\usepackage{tikz-cd} 
\usepackage{color}
\usepackage{hyperref}
\newcommand{\id}{\mathrm{id}}
\newcommand{\im}{\mathrm{im}}

\newcommand{\CA}{\mathrm{CA}}

\theoremstyle{plain}

\newtheorem{corollary}{Corollary}
\newtheorem{lemma}{Lemma}

\newtheorem{problem}{Problem}
\newtheorem{theorem}{Theorem}

\theoremstyle{definition}
\newtheorem{definition}{Definition}
\newtheorem{example}{Example}
\newtheorem{remark}{Remark}

\begin{document}

\title{Idempotent cellular automata and their natural order}
\author[1]{Alonso Castillo-Ramirez\footnote{Email: alonso.castillor@academicos.udg.mx}}
\author[1]{Maria G. Maga\~na-Chavez \footnote{Email: maria.magana3917@alumnos.udg.mx }}
\author[2]{Eduardo Veliz-Quintero \footnote{Email: eduardo.veliz9236@alumnos.udg.mx}}
\affil[1]{Centro Universitario de Ciencias Exactas e Ingenier\'ias, Universidad de Guadalajara, M\'exico.}
\affil[2]{Centro Universitario de los Valles, Universidad de Guadalajara, M\'exico.}

\maketitle

\begin{abstract}
Motivated by the search for idempotent cellular automata (CA), we study CA that act almost as the identity unless they read a fixed pattern $p$. We show that constant and symmetrical patterns always generate idempotent CA, and we characterize the quasi-constant patterns that generate idempotent CA. Our results are valid for CA over an arbitrary group $G$. Moreover, we study the semigroup theoretic natural partial order defined on idempotent CA. If $G$ is infinite, we prove that there is an infinite independent set of idempotent CA, and if $G$ has an element of infinite order, we prove that there is an infinite increasing chain of idempotent CA.    
\\

\end{abstract}

\section{Introduction}\label{intro}

This paper is about the connection between two important concepts in mathematics and computer science. On one hand, \emph{idempotence} is the property of a transformation, or an operation, of being stable after being applied once; this has been widely studied in the context of linear algebra, ring theory, semigroup theory, logic, theoretical computer science, and functional programming. On the other hand, \emph{cellular automata} (CA) are transformations of a discrete space defined by a fixed local rule that is applied homogeneously and in parallel in the whole space; they have been used in discrete complex systems modeling, and are relevant in several areas of mathematics, such as symbolic dynamics \cite{LM95}, where they are also known as \emph{sliding block codes}. Many interesting connections between the theory of CA, group theory and topology have also been studied (see the highly cited book \cite{CSC10}).        

Let $G$ be a group and let $A$ be a finite set. A \emph{configuration over $G$} is a function $x : G \to A$ and a \emph{pattern}, or \emph{block}, over a finite subset $S \subseteq G$ is a function $z : S \to A$. Denote by $A^G$ and $A^S$ the set of all configurations over $G$ and patterns over $S$, respectively. A \emph{cellular automaton} is a transformation $\tau : A^G \to A^G$ defined via a finite subset $S \subseteq G$, called a \emph{memory set} of $\tau$, and a local function $\mu : A^S \to A$. Intuitively, applying $\tau$ to a configuration $x \in A^G$ is the same as applying the local function $\mu : A^S \to A$ homogeneously and in parallel using the shift action of $G$ on $A^G$. In their classical setting, CA are studied when $G = \mathbb{Z}^d$, for $d \geq 1$ (see \cite{Kari}). The integer $d$ is usually called the dimension of the CA. 

A cellular automaton $\tau : A^G \to A^G$ is \emph{idempotent} if
\[ \tau \circ \tau = \tau,  \]
where $\circ$ is the composition of functions. Surprisingly, not much is known about idempotent CA; despite they are not interesting from a dynamical perspective, they are certainly important from an algebraic perspective. In one of the few studies we could find, Ville Salo \cite{Salo} investigates one-dimensional CA that are products of idempotents, and gives a characterization of such CA in terms of their action on periodic points. In \cite[Ex. 1.61]{ExCA}, it was noted that the image of an idempotent CA is always a \emph{subshift of finite type} (c.f. Lemma \ref{le-idem}), which means that it is a subset of $A^G$ defined via a finite set of forbidden patterns. 

The main idea of this paper is to consider CA $\tau_p^a : A^G \to A^G$ defined by a pattern $p : S \to A$ and an element $a \in A$ (see Definition \ref{CA-pattern}). We assume that the group identity $e \in G$ is in $S$. By construction, $\tau_p^a$ acts on $A^G$ almost as the identity function, except that when it reads the pattern $p$, it acts by writing the symbol $a$. This is a simple construction, yet it leads to quite subtle questions. It turns out that $\tau_p^a$ is often, but not always, an idempotent. For example, if $G =\mathbb{Z}$, $A=\{0,1\}$, and $S=\{-2,-1,0,1,2\}$, then the pattern $00010$ defines an idempotent CA, but the pattern $00001$ does not define an idempotent CA (see Table \ref{idemp}). Hence, we propose the problem of characterizing the idempotence of $\tau_p^a$ in terms of $p$ and $a$. The following result is the first step in this direction. 

\begin{theorem}
Let $G$ be a group and let $A$ be a finite set with $\vert A \vert \geq 2$. Let $S \subseteq G$ be a finite subset such that $e \in S$, and let $p : S \to A$ be a pattern. If $p$ is constant (i.e. $p(s) = p(e)$, $\forall s \in S$) or symmetrical (i.e. $S=S^{-1}$ and $p(s) = p(s^{-1})$, $\forall s \in S$), then, for any $a \in A$, the cellular automaton $\tau_p^a : A^G \to A^G$ is idempotent.
\end{theorem}

We also characterize the idempotency of $\tau_p^a$ when $p : S \to A$ is a \emph{quasi-constant pattern} (i.e., $p$ is not constant but there is $r \in S$ such that $p$ restricted to $S \setminus \{r \}$ is constant). 

\begin{theorem}\label{th-quasi-constant}
With the notation of the previous theorem, let $p : S \to A$ be a quasi-constant pattern with nonconstant term $r \in S$ and let $a \in A \setminus \{ p(e) \}$. Then, $\tau_p^a$ is idempotent if and only if one of the following holds:
\begin{enumerate}
\item $a \neq p(s)$ for all $s \in S$. 
\item $r \neq e$ and $r^2 \in S$. 
\item $r = e$ and $S = S^{-1}$. 
\end{enumerate}
\end{theorem}

In the second part of this paper, we study the natural partial order on idempotent CA. In general, for any semigroup $(\mathcal{S}, \cdot)$, the \emph{natural partial order} \cite[Sec. 1.8]{semi} defined for idempotents $\tau, \sigma \in \mathcal{S}$ is given as follows:
\[ \tau \leq \sigma \quad \Leftrightarrow \quad \tau \cdot \sigma = \sigma \cdot \tau = \tau. \]
This well-known partial order is said to be \emph{natural} because it is defined in terms of the operation of the semigroup $(\mathcal{S},\cdot)$. In 1952, Vagner generalized this to all the elements of an inverse semigroup (c.f. \cite[Sec. 7.1]{semi}), while, in 1980, Hartwig \cite{H80} and Nambooripad \cite{N80} independently generalized it to all the regular elements of a semigroup. Finally, in 1986, Mitsch \cite{M86} generalized this natural order to all the elements of any semigroup. 

The motivation behind these natural orders is that they may give information of the structure of the semigroup. In our setting, the set $\CA(G;A)$, consisting of all CA over $A^G$, is a semigroup equipped with the composition of functions. When $G$ is an infinite group such as the additive group of integers, the semigroup $\CA(G;A)$ is known to be quite intricate (e.g. it is not finitely generated \cite{Cas1, Cas2} and it contains an isomorphic copy to every finite group \cite[Theorem 6.13]{Hed}). 

In \cite[Sec. 5]{CRMC}, the natural order on idempotent \emph{elementary cellular automata} (i.e., one-dimensional cellular automata that admit a memory set $\{ -1, 0, 1\} \subseteq \mathbb{Z}$) has been studied. In this paper, for an arbitrary group $G$, we characterize the natural order of idempotent CA defined by patterns in terms of their images and kernels (see Theorem \ref{order-char}). This allows us to obtain the following result.     

\begin{theorem}
Let $G$ be a group and let $A$ be a finite set with $\vert A \vert \geq 2$.
\begin{enumerate}
\item If $G$ is infinite, there is an infinite set of independent (i.e. not pairwise comparable) idempotents in $\CA(G;A)$. 
\item If $G$ has an element of infinite order, there is an infinite increasing chain of idempotents in $\CA(G;A)$.
\end{enumerate}
\end{theorem}

The structure of this paper is as follows. In Section 2, we review some basic concepts in the theory of CA, such as the minimal memory set and the composition of two CA. In Section 3, we introduce CA defined by patterns, and study the question of their idempotency. Finally, in Section 3, we study the natural order on idempotent CA.

\section{Basic result}

In this section, we define the basic concepts on the theory of cellular automata (see \cite[Ch. 1]{CSC10}). Let $A$ be a finite set, let $G$ be a group, and let $S$ be a finite subset of $G$. We shall usually assume that the elements of $S$ are given in some order, say $S = \{s_1, \dots, s_n\}$, so we may use the following notation for a pattern $z : S \to A$:
\[  z = (z(s))_{s \in S} = z(s_1) z(s_2) \dots z(s_n).  \]

\begin{definition}
The \emph{shift action} of $G$ on $A^G$ is a function $\cdot : G \times A^G \to A^G$ defined by 
\[ (g \cdot x)(h) := x(g^{-1}h), \quad \forall x \in A^G, g, h \in G. \]
\end{definition}

The shift action is indeed a group action in the sense that $e \cdot x = x$, for all $x \in A^G$, where $e$ is the identity element of $G$, and $g \cdot (h \cdot x) = gh \cdot x$, for all $x \in A^G$, $g,h \in G$ (see \cite[p. 2]{CSC10}). 

\begin{example}
Let $G := \mathbb{Z}$ and $A := \{ 0,1 \}$. The configuration space $A^\mathbb{Z}$ may be identified with the set of bi-infinite binary sequences via the equality
\[ x = \dots x_{-2} x_{-1} x_{0} x_{1} x_{2} \dots \]
for all $x \in A^\mathbb{Z}$, where $x_{k} := x(k) \in A$. The shift action of $\mathbb{Z}$ on $A^\mathbb{Z}$ is equivalent to left and right shifts of the bi-infinite sequences. For example,
\[ 1 \cdot x = \dots x_{-3} x_{-2} x_{-1} x_{0} x_{1} \dots  \] 
For any $k \in \mathbb{Z}$, the bi-infinite sequence $k \cdot x$ is centered at $x_{-k}$. 
\end{example}

\begin{remark}
For any $g \in G$ and $x \in A^G$, we may think of $(g^{-1} \cdot x) \in A^G$ as the configuration $x$ ``centered'' in $g$, since $(g^{-1}\cdot x)(e) = x(g)$. 
\end{remark}

A \emph{subshift} $X$ of $A^G$ is a set of configurations that is $G$-invariant (i.e. $g \cdot x \in X$ for all $g \in G$, $x \in X$) and closed in the \emph{prodiscrete topology} of $A^G$ (i.e., the product topology of the discrete topology of $A$). It turns out that every subshift of $A^G$ may be defined via a set of \emph{forbidden patterns} (see \cite[Ex. 1.47]{ExCA}). A subshift is said to be of \emph{finite type} if it may be defined via a finite set of forbidden patterns. In the next section, we shall be particularly interested in subshifts $X_p \subseteq A^G$ defined by a single forbidden pattern $p : S \to A$; explicitly,
 \[ X_p := \{ x \in A^G : (g^{-1} \cdot x) \vert_S \neq p, \ \forall g \in G \}. \]

\begin{example}
For $G=\mathbb{Z}$ and $A= \{0,1\}$, consider the pattern $p : \{ 0,1\} \to A$ given by $p = 11$. The subshift $X_{11}$ is the so-called \emph{golden mean shift} \cite[Ex. 1.2.3.]{LM95}, as its entropy is the logarithm of the golden mean \cite[Ex. 4.1.4.]{LM95}.	
\end{example}

\begin{definition}[Def. 1.4.1 in \cite{CSC10}]
A \emph{cellular automaton} is a transformation $\tau : A^G \to A^G$ such that there exists a finite subset $S \subseteq G$, called a \emph{memory set} of $\tau$, and a \emph{local function}, or \emph{local rule}, $\mu : A^S \to A$, such that 
\[ \tau(x)(g) = \mu( (g^{-1} \cdot x) \vert_S ), \quad \forall x \in A^G, g \in G. \]
\end{definition}

Every cellular automaton $\tau : A^G \to A^G$ is \emph{$G$-equivariant} in the sense that $\tau(g \cdot x) = g \cdot \tau(x)$, for all $g \in G$, $x \in A^G$.

\begin{example}
Let $G := \mathbb{Z}$, $A := \{ 0,1 \}$ and $S := \{-1,0,1\}$. We identify the elements of $A^S$ with triplets in $x_{-1}x_0x_1 \in A^3$. Given a local function $\mu : A^S \to A$, the respective cellular automaton $\tau :A^\mathbb{Z} \to A^\mathbb{Z}$ is defined as follows:
\[  \tau( \dots x_{-1}.x_0x_1\dots) = \dots \mu(x_{-2}x_{-1}x_0) \mu(x_{-1}x_0x_1) \mu(x_{0}x_1x_2) \dots.   \]
In this setting, it is common to define a local function $\mu : A^S \to A$ via a table that enlists all the elements of $A^S$. For example, 
\[ \begin{tabular}{c|cccccccc}
$z\in A^S$ & $111$ & $110$ & $101$ & $100$ & $011$ & $010$ & $001$ & $000$ \\ \hline
$\mu(z) \in A$ & $0$ & $1$ & $1$ & $0$ & $1$ & $1$ & $1$ & $0$
\end{tabular}\] 
Cellular automata such as this one, that admit a memory set $S = \{ -1,0,1\} \subseteq \mathbb{Z}$, are known as \emph{elementary cellular automata} \cite[Sec. 2.5]{Kari}, and are labeled by a \emph{Wolfram number}, which is the decimal number corresponding to the second row of the defining table of $\mu : A^S \to A$. In this example, the Wolfram number of $\tau$ is 110. 
\end{example}

A memory set associated with a CA is not unique. For example, if $S \subseteq G$ is a memory set for $\tau : A^G \to A^G$ with local function $\mu : A^S \to A$, then any finite superset $S^\prime \supseteq S$ is also a memory set for $\tau$ via the local function $\mu^\prime : A^{S^\prime} \to A$ defined by $\mu^\prime(z) := \mu(z \vert_{S})$, $\forall z \in A^{S^\prime}$. However, since the intersection of memory sets for $\tau$ is a memory set for $\tau$ \cite[Lemma 1.5.1]{CSC10}, there exists a unique memory set of minimal cardinality for $\tau$, which is the intersection of all memory sets admitted by $\tau$. 

\begin{definition}[Sec. 1.5 in \cite{CSC10}]
The \emph{minimal memory set} of a cellular automaton $\tau : A^G \to A^G$ is the unique memory set of minimal cardinality admitted by $\tau$. 
\end{definition}

\begin{example}
Let $G := \mathbb{Z}$, $A := \{ 0,1 \}$ and $S := \{-1,0,1\}$. Consider the elementary cellular automaton $\tau : A^\mathbb{Z} \to A^\mathbb{Z}$ defined by the local rule $\mu :A^S \to A$ described by the following table: 
\[ \begin{tabular}{c|cccccccc}
$z\in A^S$ & $111$ & $110$ & $101$ & $100$ & $011$ & $010$ & $001$ & $000$ \\ \hline
$\mu(z) \in A$ & $0$ & $1$ & $1$ & $0$ & $0$ & $1$ & $1$ & $0$
\end{tabular}\] 
This has Wolfram number 102. A close inspection allow us to see that $\mu ( x_{-1}x_0x1) = (x_0 + x_1) \mod(2)$; hence, the minimal memory set of $\tau$ is $\{0,1\} \subseteq \mathbb{Z}$ with corresponding local rule $\mu^\prime : A^{\{ 0,1\}} \to A$ described by the following table: 
\[ \begin{tabular}{c|cccc}
$z\in A^{\{0,1 \}}$ & $11$ & $10$ & $01$ & $00$  \\ \hline
$\mu^\prime(z) \in A$ & $0$ & $1$ & $1$ & $0$
\end{tabular}\] 
\end{example}

For any two CA $\tau : A^G \to A^G$ and $\sigma : A^G \to A^G$ with memory sets $T$ and $S$, respectively, the composition $\tau \circ \sigma : A^G \to A^G$ is a CA admitting a memory set $TS := \{ ts : t \in T, s \in S \} \subseteq G$ (see \cite[Prop 1.4.9]{CSC10}). However, even if $T$ and $S$ are the minimal memory sets of $\tau$ and $\sigma$, respectively, the minimal memory set of $\tau \circ \sigma$ may be a proper subset of $TS$ \cite[Ex. 1.27]{ExCA}. If $\mu : A^T \to A$ and $\nu : A^S \to A$, are the local functions of $\tau$ and $\sigma$, respectively, then the local function of $\tau \circ \sigma$ is $\mu \star \nu : A^{TS} \to A$ given by
\[ (\mu \star \nu)( z) = \mu (\nu( z_t )_{t \in T} ), \quad \forall z \in A^{ST}, \]
where $z_t \in A^S$ is defined by $z_t(s) := z(ts)$, for all $t \in T$, $s \in S$ (\cite[Remark 1.4.10]{CSC10}. In terms of configurations, the above identity may be also written as
\[ (\mu \star \nu)( x \vert_{ST }) = \mu (\nu( (t^{-1} \cdot x) \vert_{S} )_{t \in T} ), \quad \forall x \in A^{G}. \]

\begin{example}
Let $G = \mathbb{Z}$ and $S= \{ -1,0,1 \}$. For any $\mu : A^S \to A$ and $\nu : A^S \to A$, then $\mu \star \nu : A^{S+S} \to A$, with $S + S = \{ -2, -1, 0 ,1 ,2\}$, is defined by
\[ (\mu \star \nu )( x_{-2}, x_{-1}, x_0, x_1, x_2) := \mu ( \nu( x_{-2}, x_{-1}, x_{0}) , \nu( x_{-1}, x_{0}, x_{1}), \nu ( x_{0}, x_{1}, x_{2}) ),   \]
for all $( x_{-2}, x_{-1}, x_0, x_1, x_2) \in A^{S+S}$. 
\end{example}


\section{Cellular automata defined by a pattern} 

For the rest of the paper, we shall assume that $\vert A \vert  \geq 2$, and that $\{0,1\} \subseteq A$. 

\begin{definition}\label{CA-pattern}
Let $S$ be a finite subset of $G$ such that $e \in S$ and let $a \in A$. For a pattern $p : S \to A$, define $\mu_p^a : A^S \to A$ by 
\[ \mu_p^a(z) := \begin{cases}
a & \text{ if } z = p \\
z(e) & \text{otherwise}
\end{cases}, \]
for every $z \in A^S$. Let $\tau_p^a : A^{G} \to A^G$ be the cellular automaton defined by the local function $\mu_p^a : A^S \to A$. 
\end{definition}

A generalized version of Definition \ref{CA-pattern}, which involves a finite set of patterns, was used in \cite[Ex. 1.61]{ExCA} to show that for every subshift of finite type $X \subseteq A^G$ there exists a CA $\tau : A^G \to A^G$ whose set of fixed points is precisely $X$. 

Observe that $\tau_p^a$ is the identity function if and only if $a=p(e)$. Hence, we shall assume that $a \neq p(e)$. When $A=\{0,1\}$, we simplify notation by writing $\tau_p := \tau_p^{p(e)^c}$ and $\mu_p:=\mu_p^{p(e)^c}$, where $p(e)^c$ is the complement of $p(e)$. For the rest of the paper, we assume that $S$ is a finite subset of $G$ such that $e \in S$. 

\begin{problem}
Characterize the idempotency of $\tau_p^a : A^{G} \to A^G$ in terms of the pattern $p : S \to A$ and $a \in A \setminus \{ p(e)\}$. 
\end{problem}

An elementary exploration allows us to see that in some cases $\tau_p^a$ is an idempotent and in some cases it is not. 

\begin{example}
For any group $G$, let $S:= \{ e\}$. We claim that for any pattern $p : S \to A$ and any $a \in A$, the cellular automaton $\tau_p^a : A^G \to A^G$ is idempotent. Indeed, we may identify $A^S$ with $A$, so the local rule $\mu_p^a : A^{S} \to A$ is just the transformation of $A$ that maps $p(e)$ to $a$ and fixes all other elements of $A$. The operation $\star$ defined in Section 2 is just the usual composition. Then it is easy to see that $\mu_p^a \star \mu_p^a = \mu_p^a$, so $\tau_p^a$ is an idempotent. The set
\[ \{ \mu_p^a : A \to A : p(e) \in A, a \in A\setminus \{p(e) \} \} \] 
coincides with the set of idempotent transformations $f : A \to A$ of \emph{defect} $1$ (i.e. $\vert \im(f) \vert = \vert A \vert -1$), which is known to generate the semigroup of all non-invertible transformations of $A$ \cite{Howie}. 
\end{example}

\begin{example}
Let $G := \mathbb{Z}$, $A := \{ 0,1 \}$ and $S := \{-1,0,1\}$. Recall that we identify the elements of $A^S$ with triplets in $A^3$. Consider the local rule $\mu_{010} : A^S \to A$ defined by the pattern $p = 010$:
\[ \begin{tabular}{c|cccccccc}
$z \in A^S$ & $111$ & $110$ & $101$ & $100$ & $011$ & $010$ & $001$ & $000$ \\ \hline
$\mu_{010}(z) \in A$ & $1$ & $1$ & $0$ & $0$ & $1$ & $0$ & $0$ & $0$
\end{tabular}\] 
This has Wolfram number 200, and it is an idempotent (c.f. \cite[Sec. 5]{CRMC} ).

On the other hand, the local function $\mu_{100} : A^S \to A$ defined by the pattern $p = 100$ is given by the following table:  
\[ \begin{tabular}{c|cccccccc}
$z \in A^S$ & $111$ & $110$ & $101$ & $100$ & $011$ & $010$ & $001$ & $000$ \\ \hline
$\mu_{100}(z) \in A$ & $1$ & $1$ & $0$ & $1$ & $1$ & $1$ & $0$ & $0$
\end{tabular}\] 
This has Wolfram number 220, and it is not an idempotent. 
\end{example}

\begin{example}
Not all idempotent CA are defined by a pattern: for example, if $\tau : \{0,1\}^\mathbb{Z} \to \{0,1\}^\mathbb{Z}$ is the elementary cellular automaton with Wolfram number $4$ or $223$, then $\tau$ is idempotent but there is no pattern $p : S \to A$ such that $\tau = \tau_p$.  
\end{example}

\begin{table}[!h]\centering
\caption{Idempotency of CA defined by paterns over $S \subseteq \mathbb{Z}$.}\label{idemp}
\begin{tabular}{|c|c|c|}\hline
\textbf{Domain} & \textbf{Idempotent CA} & \textbf{Non-idempotent CA} \\ \hline
$S=\{ -1,0,1 \}$ & $000$, $010$, $101$, $111$ & $001$, $011$, $100$, $110$ \\ \hline

$S = \{ 0,1,2\}$ & $000$, $010$, $101$, $111$ & $001$, $011$, $100$, $110$\\ \hline

$S= \{ -1, 0, 1, 2 \}$ & $00 0 0$, $0 0 1 0$, $0 1 0 1$, $0 1 1 0$ & $0 0 0 1$, $0 0 1 1$, $0 1 0 0$, $0 1 1 1$ \\
 & $1 0 0 1$, $1 0 1 0$, $1 1 0 1$, $1 1 1 1$ & $1 0 0 0$, $1 0 1 1$, $1 1 0 0$, $1 1 1 0$ \\ \hline
 
 $S = \{ 0, 1, 2, 3 \}$ & $0 0 0 0$, $0 1 0 0$, $0 1 1 0$, $1 0 0 1$ &  $0 0 0 1$, $0 0 1 0$, $0 0 1 1$, $0 1 0 1$, $0 1 1 1$     \\ 
& $1 0 1 1$, $1 1 1 1$ & $1 0 0 0$, $1 0 1 0$, $1 1 0 0$, $1 1 0 1$, $1 1 1 0$  \\ \hline

$ S = \{ -2, -1, 0, 1, 2 \}$ & $0 0 0 0 0$, $0 0 0 1 0$, $0 0 1 0 0$, $0 0 1 1 0$, $0 1 0 0 0$ & $0 0 0 0 1$, $0 0 0 1 1$, $0 0 1 0 1$, $0 0 1 1 1$ \\
& $0 1 0 0 1$, $0 1 0 1 0$, $0 1 1 0 0$, $0 1 1 0 1$, $0 1 1 1 0$ & $0 1 0 1 1$, $0 1 1 1 1$, $1 0 0 0 0$, $1 0 1 0 0$ \\
& $1 0 0 0 1$, $1 0 0 1 0$, $1 0 0 1 1$, $1 0 1 0 1$, $1 0 1 1 0$ & $1 1 0 0 0$, $1 1 0 1 0$, $1 1 1 0 0$, $1 1 1 1 0$ \\
& $1 0 1 1 1$, $1 1 0 0 1$, $1 1 0 1 1$, $1 1 1 0 1$, $1 1 1 1 1$ & \\ \hline

$S = \{ -1, 0, 1, 2, 3 \}$ & $0 0 0 0 0$, $0 0 0 1 1$, $0 0 1 0 0$, $0 0 1 1 0$, $0 1 0 0 1$ & $0 0 0 0 1$, $0 0 0 1 0$, $0 0 1 0 1$, $0 0 1 1 1$, $0 1 0 0 0$  \\
 & $0 1 0 1 0$, $0 1 0 1 1$, $0 1 1 0 1$, $0 1 1 1 0$, $1 0 0 0 1$ & $0 1 1 0 0$, $0 1 1 1 1$, $1 0 0 0 0$, $1 0 0 1 1$, $1 0 1 1 1$ \\
 & $1 0 0 1 0$, $1 0 1 0 0$, $1 0 1 0 1$, $1 0 1 1 0$, $1 1 0 0 1$ & $1 1 0 0 0$, $1 1 0 1 0$, $1 1 1 0 1$, $1 1 1 1 0$ \\
 & $1 1 0 1 1$, $1 1 1 0 0$, $1 1 1 1 1$ & \\ \hline
\end{tabular}
\end{table}

In Table \ref{idemp}, we determine the idempotency of $\tau_p : \{0,1\}^{\mathbb{Z}} \to \{0,1\}^{\mathbb{Z}}$ for patterns $p : S \to A$ with a given domain $S \subseteq \mathbb{Z}$. Moreover, the CA defined by patterns with domain $S= \{ -3, -2, -1, 0, 1,2,3 \}$ were studied computationally: exactly 100 of these patterns define an idempotent CA, and the rest 28 patterns define a non-idempotent CA. 

\begin{lemma}
Let $S$ be a finite subset of $G$ such that $e \in S$ and $\vert S \vert \geq 2$. For any pattern $p : S \to A$ and any $a \in A \setminus \{p(e)\}$, the minimal memory set of $\tau_p^a : A^G \to A^G$ is equal to $S$. 
\end{lemma}
\begin{proof}
It is clear by definition that $S$ is a memory set for $\tau_p^a$. In order to show that it is the minimal memory set, suppose there is a proper subset $T \subset S$ that is a memory set for $\tau_p^a$. Since $\vert S \vert \geq 2$, the definition of $\tau_p^a$ implies that it is not a constant function, so we must have $T \neq \emptyset$.  There is a local function $\mu^\prime : A^T \to A$ such that 
\[ \tau(x)(g) = \mu^\prime ( (g^{-1} \cdot x)\vert_T) = \mu_p^a( (g^{-1} \cdot x)\vert_S), \quad \forall x \in A^G, g \in G. \]
In particular, $\mu^\prime ( x\vert_T) = \mu_p^a( x\vert_S)$ for all $x \in A^G$. We divide the proof in two steps. 
\begin{itemize}
\item First we show that $e \in T$. Suppose that $e \not \in T$. Since $\vert S \vert \geq 2$, there exist $z,w \in A^S$ such that the patterns $p$, $z$ and $w$ are pairwise different and $z \vert_{S \setminus \{e\}} = w \vert_{S \setminus \{e\}}$. Since $e \not \in T$, then $z \vert_T = w \vert_T$, so $\mu^\prime(z \vert_T) = \mu^\prime(w \vert_T)$. However, since $p \neq z$ and $p \neq w$, the definition of $\mu_p^a$ implies that
\[ \mu_p^a(z) = z(e) \neq w(e) = \mu_p^a(w). \] 
This is a contradiction, so $e \in T$. 

\item Now we show that if $s \in S$, $s \neq e$, then $s \in T$. Suppose that $s \not \in T$ and let $z \in A^S$ be such that $z \vert_{S \setminus \{s\}} = p \vert_{S \setminus \{s\}}$ but $z(s) \neq p(s)$. In particular, $z(e) = p(e)$. As $s \not \in T$, then $p \vert_T = z \vert_T$ so $\mu^\prime( p \vert_T) = \mu^\prime( z\vert_T)$, but 
\[ \mu_p^a( p ) = a \neq p(e) = z(e) = \mu_p^a(z). \] 
This is a contradiction, so the result follows. 
\end{itemize}
\end{proof}

In a similar spirit as the previous lemma, \cite{MMS} examines the minimal memory set of cellular automata generated by a finite set of patterns. Recall that $X_p$ is the subshift of $A^G$ defined by the forbidden pattern $p : S \to A$. 

\begin{lemma}\label{le-idem}
Let $\tau_p^a : A^G \to A^G$ be the cellular automaton defined by a pattern $p : S \to A$ and $a \in A \setminus \{ p(e) \}$. Then,
\[ X_p \subseteq \im(\tau_p^a), \]
with equality if and only if $\tau_p^a$ is idempotent. 
\end{lemma}
\begin{proof}
It is easy to show that $X_p$ is equal to the set of fixed points of $\tau_p^a$; this is,
\[ X_p = \text{Fix}(\tau_p^a) := \{x \in A^G : \tau_p^a(x) = x \}. \]
Hence, it follows that $X_p \subseteq \im(\tau_p^a)$. The last statement follows because $\tau_p^a$ is idempotent if and only if $\text{Fix}(\tau_p^a) = \im(\tau_p^a)$.
\end{proof}

\begin{lemma}\label{le-2}
Let $\tau_p^a : A^G \to A^G$ be the cellular automaton defined by a pattern $p : S \to A$ and $a \in A \setminus \{p(e)\}$. Then, $\tau_p^a$ is not idempotent if and only if there exists $x \in A^G$ such that 
\begin{equation} \label{eq-x}
\mu_p^a( (s^{-1} \cdot x) \vert_S ) = p(s), \quad \forall s \in S. 
\end{equation}
Furthermore, if such $x \in A^G$ exists, it must also satisfy that
\begin{enumerate}
\item $x \vert_S \neq p$,
\item $x(e) = p(e)$. 
\item $(t^{-1} \cdot x) \vert_S = p$ for some $t \in S - \{e\}$. 
\end{enumerate}
\end{lemma}
\begin{proof}
By Lemma \ref{le-idem}, $\tau_p^a$ is not idempotent if and only if $X_p \subsetneq \im(\tau_p^a)$, which means that the there exists $x \in A^G$ such that $p$ is a subpattern of $\tau_p^a(x)$. Since $\tau_p^a$ is $G$-equivariant, we may assume $\tau_p^a(x) \vert_S = p$. In term of its local rule, this is equivalent to equation (\ref{eq-x}). 

In particular, taking $s=e$, we obtain $\mu_p^a( x \vert_S) = p(e)$, which implies that $x \vert_S \neq p$ (as otherwise $\mu_p^a( x \vert_S) = a \neq p(e)$). By the definition of $\mu_p^a$, we have 
\[  \mu_p^a( x \vert_S) = x(e) = p(e). \]

Parts (1.) and (2.) follow. For part (3.), suppose that $(s^{-1} \cdot x) \vert_S \neq p$ for all $s \in S$. By definition of $\mu_p^a$, this implies that for all $s \in S$, 
\[ p(s) = \mu_p^a((s^{-1} \cdot x) \vert_S) = (s^{-1} \cdot x)(e) = x(s),   \]
which contradicts that $x \vert_S \neq p$. Therefore, there exists $t \in S \setminus \{ e \}$ such that $(t^{-1} \cdot x) \vert_S = p$. 
\end{proof}

\begin{lemma}\label{le-main}
Let $S$ be a finite subset of $G$ with $\vert S \vert \geq 2$. Let $p : S \to A$ and suppose there exists $a \in A \setminus \{ p(e)\}$ such that $a \neq p(s)$ for all $ s \in S$. Then $\tau_p^a : A^G \to A^G$ is idempotent.  
\end{lemma}
\begin{proof}
Suppose that $\tau_p^a$ is not idempotent, so there exists $x \in A^G$ satisfying the conditions of Lemma \ref{le-2}. In particular, there exists $t \in S \setminus \{e\}$ such that $(t^{-1} \cdot x) \vert_S = p$. By equation (\ref{eq-x}) and the definition of $\mu_p^a$, 
\[ p(t) = \mu_p^a( (t^{-1} \cdot x) \vert_S )  = a, \] 
which is a contradiction with the hypothesis. 
\end{proof}

\begin{corollary}\label{const-pattern}
Let $p : S \to A$ be a constant pattern (i.e., $p(s) = p(t)$, for all $s,t \in S$) and let $a \in A \setminus \{p(e)\}$. Then $\tau_p^a : A^G \to A^G$ is idempotent.
\end{corollary}

\begin{corollary}
For any finite subset $S \subseteq G$ with $e \in S$, there exists an idempotent cellular automaton $\tau : A^G \to A^G$ whose minimal memory set is equal to $S$. 
\end{corollary}

A subset $S \subseteq G$ is \emph{closed under inverses} if $s^{-1} \in S$, for all $s \in S$; in such case, we write $S = S^{-1}$. A pattern $p : S \to A$ is called \emph{symmetrical} if $S = S^{-1}$ and $p(s) = p(s^{-1})$ for all $s \in S$.  

\begin{lemma}\label{le-sym}
Let $p : S \to A$ be a symmetrical pattern and $a \in A \setminus \{p(e)\}$. Then, $\tau_p^a : A^G \to A^G$ is idempotent. 
\end{lemma}
\begin{proof}
Suppose that $\tau_p^a$ is not idempotent, so there exists $x \in A^G$ satisfying the conditions of Lemma \ref{le-2}. Let $t \in S \setminus \{e \}$ be such that $(t^{-1} \cdot x) \vert_S = p$. Evaluating this on $t^{-1} \in S$ we obtain that
\[(t^{-1} \cdot x)(t^{-1}) = x(tt^{-1}) = x(e) = p(t^{-1}).  \]
By symmetry, $p(t^{-1}) = p(t)$ and by part (2.) of Lemma \ref{le-2}, $x(e) = p(e)$. Therefore, 
\[  p(e) = x(e) = p(t^{-1}) = p(t). \]
However, by equation (\ref{eq-x}) and the definition of $\mu_p^a$,
\[ p(t) =   \mu_p^a( (t^{-1} \cdot x) \vert_S ) = a \neq p(e), \] 
which is a contradiction. The result follows. 
\end{proof}

Our aim now is to characterize the idempotency of the particular kind of patterns introduced by the next definition. 

\begin{definition}
A pattern $p : S \to A$ is called \emph{quasi-constant} if it is nonconstant and there exists $r \in S$ such that $p$ restricted to $S \setminus \{r\}$ is constant. In such case, we say that $r \in S$ is the \emph{nonconstant term} of $p$. 
\end{definition}

Theorem \ref{th-quasi-constant} follows by Lemma \ref{le-main}, together with the two following results.

\begin{lemma}
Let $p : S \to A$ be a quasi-constant pattern with nonconstant term $r \in S \setminus \{e\}$ such that $p(r) = a \in A \setminus \{ p(e) \}$. Then, $\tau_p^a$ is idempotent if and only if $r^2 \in S$. 
\end{lemma}
\begin{proof}
Suppose that $r^2 \in S$ and that $\tau_p^a$ is not idempotent. Hence, there exists $x \in A^G$ satisfying the conditions of Lemma \ref{le-2}. Let $t \in S \setminus \{ e\}$ be as in part (3.) of Lemma \ref{le-2}, so $(t^{-1} \cdot x ) \vert_S = p$. By equation (\ref{eq-x}) of  Lemma \ref{le-2}, we have
\[ \mu_p^a ( (t^{-1} \cdot x) \vert_S)  =  p(t) = a = p(r). \]
Since $p$ is quasi-constant with nonconstant term $r$, we must have $t=r$. This implies that $(r^{-1} \cdot x)(s) = p(s)$ for all $s \in S$, so, in particular, $x(r^2) = p(r)$. Evaluating equation (\ref{eq-x}) in $s=r^2$, we get 
\[ \mu_p^a( (r^{-2} \cdot x) \vert_S ) = p(r^2). \]
We have two cases:
\begin{itemize}
\item If $(r^{-2} \cdot x) \vert_S = p$, then, $\mu_p^a( (r^{-2} \cdot x) \vert_S ) = a = p(r)$.
\item If $(r^{-2} \cdot x) \vert_S \neq p$, then $\mu_p^a( (r^{-2} \cdot x) \vert_S ) = (r^{-2} \cdot x)(e) = x(r^2) = p(r)$. 
\end{itemize}
In any case, we obtain that $p(r^2) = p(r)$, which cannot happen as $p$ is quasi-constant with nonconstant term $r$ and $r \neq r^2$ (since $r \neq e$). 

For the converse, suppose that $r^2 \not \in S$. Define $x \in A^G$ such that $x(g) = p(e)$ for all $g \in S^2 \setminus \{ r^2\}$ and $x(r^2) = p(r)$. Observe that for all $s \in S \setminus \{r\}$, we have that $(s^{-1} \cdot x) \vert_S \neq p$, because 
\[ (s^{-1} \cdot x)(r) = x(sr) = p(e) \neq p(r).     \]
Therefore, 
\[ \mu_p^a ( (s^{-1} \cdot x)\vert_S) = (s^{-1} \cdot x)(e) = x(s) = p(e) = p(s), \quad \forall s \in S \setminus \{r, r^2\}. \] 
Since $r^2 \not \in S$, the above follows for all $s \in S \setminus \{r\}$. 
 
Furthermore, $(r^{-1} \cdot x) \vert_S = p$ because $(r^{-1} \cdot x)(s) = x(rs) = p(e) = p(s)$ for all $s \in S \setminus \{r\}$ and $(r^{-1} \cdot x)(r) = x(r^2) = p(r)$. Therefore,
\[  \mu_p^a ( (r^{-1} \cdot x) \vert_S) =  a = p(r).     \]
This shows that equation (\ref{eq-x}) of Lemma \ref{le-2} is satisfied, so $\tau_{p}^a$ is not idempotent.  
\end{proof}

\begin{lemma}
Let $p : S \to A$ be a quasi-constant pattern with nonconstant term $e \in S$ such that $p(s) = a \in A \setminus \{ p(e) \}$ for all $s \in S \setminus \{e\}$. Then, $\tau_p^a$ is idempotent if and only if $S = S^{-1}$.
\end{lemma}
\begin{proof}
If $S = S^{-1}$, then $p$ is symmetric, so the result follows by Lemma \ref{le-sym}. Suppose that $S$ is not closed under inverses, so there exists $k \in S$ such that $k^{-1} \not \in S$. Clearly, $k \neq e$. Define $x \in A^G$ such that $x(e) = x(k) = p(e)$ and $x(g) = p(k)$ for all $g \in S^2 \setminus \{ e,k \}$. Observe that for all $s \in S \setminus \{e, k \}$ we have $(s^{-1} \cdot x) \vert_S \neq p$ because $(s^{-1} \cdot x)(e) = x(s)  = p(k) \neq p(e)$. Hence, 
\[ \mu_p^a( (s^{-1} \cdot x) \vert_S ) =  x(s) = p(k) = p(s), \quad  \forall s \in S \setminus \{ e,k \}.  \]
Now, $x \vert_S \neq p$ because $x(k) = x(e) = p(e) \neq p(k)$. Hence,
\[ \mu_p^a(x \vert_S) = x(e) = p(e). \]
Finally, we shall show that $(k^{-1} \cdot x) \vert_S = p$. For all $s \in S \setminus \{ e\}$, then $ks \neq k$ and $ks \neq e$ (as $k^{-1} \not \in S$). Thus, for all $s \in S \setminus \{ e\}$ we have $(k^{-1} \cdot x)(s) = x(ks) = p(k) = p(s)$. Moreover, $(k^{-1} \cdot x) (e) = x(k) = p(e)$. Hence, using the hypothesis that $a = p(s)$ for all $ s \in S \setminus \{e\}$, we deduce
\[ \mu_p^a( (k^{-1} \cdot x) \vert_S) = a = p(k).    \] 
This shows that equation (\ref{eq-x}) of Lemma \ref{le-2} is satisfied, so $\tau_{p}^a$ is not idempotent.  
\end{proof}


\section{The natural order on idempotent CA} 

Recall that the natural order defined on two idempotents $\tau, \sigma \in \CA(G;A)$ is given by
\[ \tau \leq \sigma \quad \Leftrightarrow \quad  \tau \sigma = \sigma \tau = \tau, \]
where $\tau \sigma = \tau \circ \sigma$ is the composition. It is clear that the identity $\id : A^G \to A^G$ is the maximal idempotent in $\CA(G;A)$, while the minimal idempotents are the constant cellular automata $\sigma_a : A^G \to A^G$, with $a \in A$, defined by $\sigma_a(x) := a^G \in A^G$, where $a^G(g) = a$, for all $g \in G$. 

The following result is a generalization of \cite[Lemma 5]{CRMC}.

\begin{lemma}
Let $\tau : A^G \to A^G$ be an idempotent cellular automaton. For any $a \in A$, $\sigma_a \leq \tau$ if and only if $\tau( a^G) = a^G$. 
\end{lemma}
\begin{proof}
Since $\sigma_a$ is constant, then $\sigma_a \tau = \sigma_a$. Now, for any $x \in A^G$, $\tau \sigma_a(x) = \tau( a^G)$. Hence, $\tau(a^G) = a^G$ if and only if $\tau \sigma_a = \sigma_a$, which holds if and only if $\sigma_a \leq \tau$. 
\end{proof}

\begin{corollary}
Let $\tau_p^a: A^G \to A^G$ be an idempotent cellular automaton defined by a pattern $p : S \to A$ and $a \in A \setminus \{p(e) \}$. For any $b \in A$, $\sigma_b \leq \tau_p^a$ if and only if $p \neq b^S$, where $b^S : S \to A$ is the pattern defined by $b^S(s) = b$, for all $s \in S$. 
\end{corollary}
\begin{proof}
Observe that $\tau_p^a(b^G) = b^G$ if and only if $p \neq b^S$, so the result follows by the previous lemma. 
\end{proof}

For $\tau \in \CA(G;A)$, define the \emph{kernel} of $\tau$ as the following equivalence relation: 
\[ \ker(\tau) := \{ (x,y) \in A^G \times A^G : \tau(x) = \tau(y) \}.  \]
The following result is well-known for transformation semigroups (see \cite[Prop.2.3]{KM86}), but we shall add its proof for completeness.

\begin{lemma}\label{le-im-ker}
Let $\tau, \sigma \in \CA(G;A)$ be idempotents. 
\begin{enumerate}
\item If $\tau = \sigma \tau$, then $\im(\tau) \subseteq \im(\sigma)$.
\item $\tau = \tau\sigma$ if and only if $\ker(\sigma) \subseteq \ker(\tau)$.
\end{enumerate}
\end{lemma}
\begin{proof}
Let $\tau(x) \in \im(\tau)$. Since $\tau = \sigma \tau$, then $\tau(x) = \sigma \tau(x) \in \im(\sigma)$, so $\im(\tau) \subseteq \im(\sigma)$. Part (1.) follows. 

For part (2.), suppose that $\tau = \tau\sigma$ and let $(x,y) \in \ker(\sigma)$, so $\sigma(x) = \sigma(y)$. Applying $\tau$ on both sides we obtain
\[\tau \sigma(x) = \tau\sigma(y) \quad \Rightarrow \quad \tau(x) = \tau(y).  \]
Therefore, $(x,y) \in \ker(\tau)$, so $\ker(\sigma) \subseteq \ker(\tau) $.

Conversely, suppose that $\ker(\sigma) \subseteq \ker(\tau)$. Since $\sigma$ is idempotent, we have that for any $x \in A^G$, $(\sigma(x), x) \in \ker(\sigma)$. By hypothesis, $(\sigma(x), x) \in \ker(\tau)$, so $\tau\sigma(x) = \tau(x)$ for all $x \in A^G$. The result follows. 
\end{proof}

The following is a characterization of the natural partial order on idempotent CA defined by patterns in terms of images and kernels.  

\begin{theorem}\label{order-char}
Let  $p : S_1 \to A$ and $q : S_2 \to A$ be patterns such that $\tau_p^a$ and $\tau_q^b$ are idempotents, for some $a \in A \setminus \{p(e)\}$ and $b \in A \setminus \{q(e)\}$. Then, $\tau_p^a \leq \tau_q^b$ if and only if $X_p \subseteq X_q$ and $\ker(\tau_q^b) \subseteq \ker(\tau_p^a)$. 
\end{theorem}
\begin{proof}
Recall that, by Lemma \ref{le-idem}, $\im(\tau_p^a) = X_p$ and $\im(\tau_q^b) = X_q$. Lemma \ref{le-im-ker} implies the direct implication. For the converse, suppose that $X_p \subseteq X_q$ and $\ker(\tau_q^b) \subseteq \ker(\tau_p^a)$. By Lemma \ref{le-im-ker} (2.), $\tau_p^a = \tau_p^a \tau_q^b$. Now, $\im(\tau_p^a)=X_p \subseteq X_q$ implies that the pattern $q$ never appears as a subpattern in $\im(\tau_p^a)$, so $\tau_q^b$ acts as the identity over $\im(\tau_p^a)$. Therefore, $\tau_q^b \tau_p^a = \tau_p^a$, and $\tau_p^a \leq \tau_q^b$. 
\end{proof}

\begin{corollary}
Let  $p : S_1 \to A$ and $q : S_2 \to A$ be patterns such that $\tau_p^a$ and $\tau_q^b$ are idempotents, for some $a \in A \setminus \{p(e)\}$ and $b \in A \setminus \{q(e)\}$. If $\tau_p^a < \tau_q^b$, then $X_p \subset X_q$. 
\end{corollary}
\begin{proof}
If $\tau_p^a < \tau_q^b$, then $\tau_q^b \tau_p^a = \tau_p^a \tau_q^b = \tau_p^a$. Now, if $X_p = X_q$, then $X_q \subseteq X_p$ and the proof of the previous theorem implies that $\tau_p^a \tau_q^b = \tau_q^b$. Therefore, $\tau_p^a = \tau_q^b$, which contradicts the hypothesis. 
\end{proof}

We define an order on patterns. For $p : S_1 \to A$ and $q : S_2 \to A$, say that $p \leq q$ if and only if $S_1 \subseteq S_2$ and $p = q \vert_{S_1}$. Observe that if $p \leq q$, then $X_p \subseteq X_q$.

\begin{lemma}\label{le-comp}
Let  $p : S_1 \to A$ and $q : S_2 \to A$ be patterns such that $\tau_p^a$ and $\tau_q^b$ are idempotents, for some $a \in A \setminus \{p(e)\}$ and $b \in A \setminus \{q(e)\}$. Suppose that $q(e) \neq a$ and that $\tau_p^a \leq \tau_q^b$. Then, $a=b$ and $p \leq q$.
\end{lemma}
\begin{proof}
Let $x \in A^G$ be such that $x \vert_{S_2} = q$ and let $y:= \tau_q^b(x)$. Then,
\[ y = \tau_q^b (x) = \tau_q^b(\tau_q^b(x)) = \tau_q^b(y). \]
Evaluating on $e$, we obtain
\begin{equation}\label{eq-muq}
y(e) = \mu_q^b( x \vert_{S_2}) = b \neq q(e) =x(e). 
\end{equation}
Since $\ker(\tau_q^b) \subseteq \ker(\tau_p^a)$, we have that $\tau_p^a(x) = \tau_p^a(y)$, so
\[ \mu_p^a( x \vert_{S_1}) = \mu_p^a( y \vert_{S_1}). \]
Since $y(e) \neq x(e)$, then $x \vert_{S_1} \neq y \vert_{S_1}$. We have three cases:
\begin{enumerate}
\item If $x \vert_{S_1} \neq p$ and $y \vert_{S_1} \neq p$, then
\[ x(e)= \mu_p^a( x \vert_{S_1}) = \mu_p^a( y \vert_{S_1}) = y(e), \]
which contradicts (\ref{eq-muq}).

\item If  $x \vert_{S_1} \neq p$ and $y \vert_{S_1} = p$, then
\[ x(e) = \mu_p^a( x \vert_{S_1}) = \mu_p^a( y \vert_{S_1}) = a. \]
Hence, as $x \vert_{S_2} = q$, then $q(e) = x(e)= a$, which contradicts the hypothesis.  

\item If $x \vert_{S_1} = p$ and $y \vert_{S_1} \neq p$, then 
\[ a = \mu_p^a( x \vert_{S_1})  = \mu_p^a( y \vert_{S_1}) = y(e) = b. \]
\end{enumerate}
As (3.) is the only case without contradiction, this shows that $a=b$ and $x \vert_{S_1} = p$. Hence, for every $x \in A^G$, if $x \vert_{S_2} = q$, then $x \vert_{S_1} = p$. Suppose there exists $s \in S_1 \setminus S_2$, and take $x \in A^G$ such that $x \vert_{S_2} = q$ and $x(s) \neq p(s)$. Then $x \vert_{S_1} \neq p$, which contradicts the previous property. Therefore, $S_1 \subseteq S_2$ and $q \vert_{S_1} = (x \vert_{S_2}) \vert_{S_1}= p$, so $p \leq q$. 
\end{proof}

\begin{corollary}
Let $G$ be an infinite group. Then $\CA(G;A)$ has an infinite set of independent idempotents.
\end{corollary}
\begin{proof}
Let $(g_i)_{i \in \mathbb{N}}$ be an infinite sequence of different nontrivial elements of $G$ such that $g_i \not\in \{ g_j, g_j^{-1} : j < i \}$. For each, $i \in \mathbb{N}$, define $S_i := \{ e, g_j, g_j^{-1} : j \leq i \}$ and let $p_i : S_i \to A$ be the pattern defined by 
\[ p_i(e) = p_i(g_j) = p_i(g^{-1}_j) = 0, \ \forall j < i, \quad \text{ and } \quad p_i(g_i) = p_i(g^{-1}_i) = 1.  \]
Since $p_i$ is symmetrical, $\tau_{p_i}^1$ is idempotent for all $i \in \mathbb{N}$, by Lemma \ref{le-sym}. For all $i, k \in \mathbb{N}$, we have $p_i(e) = p_k(e) = 0 \neq 1$, and $p_i$ and $p_k$ are not comparable in the order of patterns. By Lemma \ref{le-comp}, $\tau_{p_i}^1$ and $\tau_{p_k}^1$ are not comparable for all $i, k \in \mathbb{N}$.
\end{proof}

We denote by $\vert s \vert$ the \emph{order} of $s \in G$, which is the least positive integer $k$ such that $s^k =e$, or $\infty$ in case no such $k$ exists. The converse of Lemma \ref{le-comp} is not true, as it is shown by the following example.

\begin{example}
Suppose that $G$ has an element $s \in G$ such that $\vert s \vert > 3$. Let 
\[ S_1 :=\{e,s\} \text{ and } S_2 :=  \{ e, s, s^{-1} \} \]
and consider the constant patterns $p :  S_1 \to A$ and $q : S_2 \to A$ defined by $p = 00$ and $q=000$. We claim that $\ker(\tau_q^1)$ is not contained in $\ker(\tau_p^1)$. Consider the configuration 
\[ x(g) := \begin{cases}
0 & \text{ if } g \in S_2 \\
1 & \text{ otherwise.} 
 \end{cases} \]
Let $y := \tau_q^1(x)$. Clearly, $(x,y) \in \ker(\tau_q^1)$. Observe that $(s \cdot x) \vert_{S_2} \neq q$, because $(s \cdot x)(s^{-1}) = x(s^{-2}) = 1$,  since $\vert s \vert > 3$, so $s^{-2} \not \in S_2$. Hence
\begin{align*}
y(s^{-1}) & = \tau_q^1(x)(s^{-1}) =  \mu_q^1((s \cdot x) \vert_{S_2}) = x(s^{-1}) = 0.  \\
y(e) & =  \tau_q^1(x)(e) =   \mu_q^1(x \vert_{S_2}) =  1. 
\end{align*}
This implies that $(s \cdot y) \vert_{S_1} \neq p$, so 
\[ \tau_p^1(y)(s^{-1}) = \mu_p^1( (s \cdot y) \vert_{S_1}) = y(s^{-1}) = 0.  \]
However, $(s \cdot x) \vert_{S_1} = p$, so
\[ \tau_p^1(x)(s^{-1}) = \mu_p^1( (s \cdot x) \vert_{S_1}) = 1. \]
This shows that $\tau_p^1(y) \neq\tau_p^1(x)$, so $(x,y) \not \in \ker(\tau_p^1)$. By Theorem \ref{order-char}, $\tau_p^1 \not \leq \tau_q^1$. 
\end{example}

However, the constant patterns may be used to define an infinite increasing chain of idempotents in $\CA(G;A)$ if the domains of the patterns do not include inverses. This is the key idea in the next proof.   

\begin{theorem}
Suppose that $G$ contains an element of infinite order. Then, there is an infinite increasing chain of idempotents in $\CA(G;A)$.
\end{theorem}
\begin{proof}
Let $s \in G$ be an element of infinite order. For $i \in \mathbb{N}$, define $S_i := \{e, s, \dots, s^i \}$. Define constant patterns $p_i : S_i \to A$ by $p_i(s^k) = 0$, for all $s^k \in S_i$. We will show that $\tau_{p_i}^1 \leq \tau_{p_{i+1}}^1$, for all $i \in \mathbb{N}$. 

Clearly, $p_i \leq p_{i+1}$, so $X_{p_i} \subseteq X_{p_{i+1}}$. We will show that $\tau_{p_i}^1 = \tau_{p_i}^1 \circ \tau_{p_{i+1}}^1$, so we may conclude by Lemma \ref{le-im-ker} and Theorem \ref{order-char} that $\tau_{p_i}^1 \leq \tau_{p_{i+1}}^1$.

Let $\tau_i := \tau_{p_i}^1$ and $\mu_i := \mu_{p_i}^1$. Observe that $\tau_i = \tau_i \tau_{i+1}$ holds if and only if, for all $x \in A^G$,
\begin{equation} \label{eq-ex}
 \mu_{i}( x \vert_{S_i}) = \mu_{i}( \mu_{i+1}( (s^{-k} \cdot x) \vert_{S_{i+1}})_{s^k \in S_i}  ). 
\end{equation}
We have a few cases.
\begin{itemize}
\item $x \vert_{S_i} = p_i$. In this case, we have
\[ \mu_i( x \vert_{S_i}) =  1.  \]

\begin{itemize}
\item[a)] Suppose that $x \vert_{S_{i+1}} = p_{i+1}$. Then $\mu_{i+1}( x \vert_{S_{i+1}}) = 1$, so the right-hand-side (RHS) of (\ref{eq-ex}) must be equal to $1$. 

\item[b)] Suppose that $x \vert_{S_{i+1}} \neq p_{i+1}$. Since $x(s^k) = 0$, for all $k \in \{0, \dots, i\}$, we must have $x(s^{i+1}) \neq 0$. Hence, for all $s^k \in S_i$, we have $(s^{-k} \cdot x) \vert_{S_{i+1}} \neq p_{i+1}$, since $(s^{-k} \cdot x)(s^{i+1-k}) = x(s^{i+1}) \neq 0 = p_{i+1}(s^{i+1-k})$. Thus, for all $s^k \in S_i$,
\[ \mu_{i+1}( (s^{-k} \cdot x) \vert_{S_{i+1}}) = (s^{-k} \cdot x)(e) = x(s^k) = 0.     \]
Therefore, the RHS of (\ref{eq-ex}) must be equal to $1$, as the input for $\mu_i$ there is $p_i$. 
\end{itemize}
\item $x \vert_{S_i} \neq p_i$. In this case, we have
\[ \mu_{i}( x \vert_{S_i}) = x(e).  \]
\begin{itemize}
\item[c)] Suppose that $x(e) = 0$. Since $x \vert_{S_i} \neq p_i$, then $x(s^k) \neq 0$ for some $s^k \in S_i$. This implies that $(s^{-k} \cdot x) \vert_{S_{i+1}} \neq p_{i+1}$ (because $(s^{-k} \cdot x)(e)=x(s^k)\neq 0$), so 
\[ \mu_{i+1}( (s^{-k} \cdot x) \vert_{S_{i+1}}) =(s^{-k} \cdot x)(e) = x(s^k)\neq 0. \]
Therefore, the input $\mu_i$ in the RHS of (\ref{eq-ex}) is not equal to $p_i$. As $x \vert_{S_{i+1}} \neq p_{i+1}$ (because $x(s^k) \neq 0$), the RHS of (\ref{eq-ex}) must be equal to
\[  \mu_{i+1}(  x \vert_{S_{i+1}}) = x(e) = 0.  \] 

\item[d)] Suppose that $x(e) \neq 0$. Hence, $x \vert_{S_{i+1}} \neq p_{i+1}$, so $\mu_{i+1}( x \vert_{S_{i+1}})  = x(e) \neq 0$. Therefore, the RHS of (\ref{eq-ex}) must be equal to $x(e)$, as the input for $\mu_{i}$ there is not equal to $p_i$. 
\end{itemize}
\end{itemize}
Therefore, equation (\ref{eq-ex}) holds in all the cases, and the result follows. 
\end{proof}


\section*{Acknowledgments}

The second and third authors were supported by CONAHCYT \emph{Becas nacionales para estudios de posgrado}.


\end{document}